\documentclass[reqno]{amsart}
\usepackage{amssymb,enumerate,mathrsfs}

\numberwithin{equation}{section}

\newtheorem{theorem}[equation]{Theorem}
\newtheorem{proposition}[equation]{Proposition}
\newtheorem{lemma}[equation]{Lemma}

\theoremstyle{definition}
\newtheorem{definition}[equation]{Definition}

\newtheorem{example}[equation]{Example}

\def\C{\mathbb C}
\def\Dom{\mathcal D}
\def\K{\mathcal K}
\def\L{\mathscr L}
\def\R{\mathbb R}
\def\Sing{\mathcal E}

\def\eps{\varepsilon}
\def\Gr{\mathrm{Gr}}
\def\minus{\backslash}
\def\open#1{\smash[t]{\overset{{}_{\circ}}{#1}{}}}
\def\set#1{\{#1\}}
\def\st{:}

\DeclareMathOperator{\bgres}{bg-res}
\DeclareMathOperator{\bgspec}{bg-spec}
\DeclareMathOperator{\Ind}{ind}
\DeclareMathOperator{\rg}{rg}
\DeclareMathOperator{\spec}{spec}
\DeclareMathOperator{\rk}{rk}

\begin{document}
\title{A conic manifold perspective of elliptic operators on graphs}

\author{Juan B. Gil}
\address{Penn State Altoona\\ 3000 Ivyside Park \\ Altoona, PA 16601-3760}
\email{jgil@psu.edu}
\author{Thomas Krainer}
\address{Penn State Altoona\\ 3000 Ivyside Park \\ Altoona, PA 16601-3760}
\email{krainer@psu.edu}
\author{Gerardo A. Mendoza}
\address{Department of Mathematics\\ Temple University\\ Philadelphia, PA 19122}
\email{gmendoza@math.temple.edu}

\begin{abstract}
We give a simple, explicit, sufficient condition for the existence of
a sector of minimal growth for second order regular singular 
differential operators on graphs.  We specifically consider operators 
with a singular potential of Coulomb type and base our analysis on the 
theory of elliptic cone operators.  
\end{abstract}


\maketitle

%%%%%%%%%%%%%%%%%%%%%%%%%%%%%%%%%%%%%%%%%%%%%%%%%%%%%%%%%%%%%%%%%%%%%%
\section{Introduction}

The analysis of differential operators on graphs is an area of current
interest with a long tradition (cf. \cite{vonBelow1,KostrykinSchrader1,
KostrykinSchrader2,SavchukShkalikov}, in particular the survey article
\cite{Kuchment2}).

In this paper we adopt the point of view that a graph is a one-dimensional 
manifold with conical singularities, and that a differential operator on 
it with smooth coefficients and regular singular points at most at the 
endpoints of the edges is a cone differential operator.  
Among these, the simplest ones are second order operators whose
coefficients are smooth up to the boundary, except for the potential term,
which is permitted to have Coulomb type singularities.
This class of operators is interesting not just intrinsically but also
because it is already one for which the notion of boundary values is not
trivial. The aim of the present paper is to analyze in detail the
existence of sectors of minimal growth for the latter class of operators, 
the main point being that this can be done explicitly and in simple terms. 
We follow the approach developed in \cite{GKM1,GKM2,GKM3} on general elliptic
cone operators.
At the end of the paper we shall indicate how the general situation on a 
graph can be analyzed with the same tools.

Recall that a closed sector
\begin{equation*}
\Lambda = \{re^{i\varphi}\st r \geq 0,\; |\varphi - \varphi_0| \leq a\} 
\subset \C
\end{equation*}
is a sector of minimal growth for the closed operator
\begin{equation}\label{TheOperator}
A : \Dom \subset H \to H
\end{equation}
acting in the Hilbert space $H$ if and only if $A - \lambda : \Dom \to H$ 
is invertible for $\lambda \in \Lambda$ with $|\lambda| > 0$ sufficiently
large with the resolvent satisfying the estimate
\begin{equation}\label{Decay}
\|(A-\lambda)^{-1}\|_{\L(H)} = O(|\lambda|^{-1})
\quad\text{as } |\lambda| \to \infty.
\end{equation}

The existence of sectors of minimal growth is directly related to the 
well posedness and maximal regularity for parabolic evolution equations, 
see e.g. \cite{EngelNagel}.   
Going further into detail, the asymptotic behavior of resolvents provides
insight about heat trace asymptotics, zeta functions, and other
spectral invariants.

Consider a disjoint union of closed intervals,
\begin{equation*}
G = \bigsqcup_{j=1}^N E_j.
\end{equation*}
In particular, $G$ is a one-dimensional compact manifold with boundary $\partial G$. A relation of equivalence on $\partial G$, extended trivially to all of $G$, gives rise to a graph $\Gamma=G/{\sim}$ with vertices $V=\partial G/{\sim}$ and edges $E_j$. Let $\open \Gamma=\Gamma\minus V$. The canonical diffeomorphism $\open G\to\open \Gamma$ allows us to conduct analysis on $\Gamma$.

As indicated above, we focus on second order differential operators
\begin{equation}\label{AGlobal}
A : C_c^{\infty}(\open \Gamma) \to C_c^{\infty}(\open \Gamma)
\end{equation}
which on each $E_j \cong [-1,1]$, $j=1,\ldots,N$, are of the form
\begin{equation}\label{ALocalEdge}
A_j = a_j D_s^2 + b_j D_s + \frac{c_j}{(1-s)(1+s)} : C_c^{\infty}(-1,1)
\to
C_c^{\infty}(-1,1),
\end{equation}
where $D_s = \frac{1}{i}\partial_s$ and $s$ is the variable in $[-1,1]$. The coefficients $a_j$, $b_j$, and $c_j$ are assumed to be functions on $[-1,1]$, smooth up to the endpoints. In other words, as already mentioned, the potential terms are permitted to have Coulomb singularities at the endpoints, and so the differential operator $A$ may be singular at the vertices of the graph. Furthermore, we assume $A$ to be elliptic, i.e., in \eqref{ALocalEdge} we assume $a_j(s)\not=0$ on $[-1,1]$.

Let $L^2(G)$ be the $L^2$-space with respect to (any) smooth metric on $G$. With the push-forward measure we have $L^2(\Gamma)\cong L^2(G)$. Every natural closed extension of
\begin{equation*}
A : C_c^\infty(\open \Gamma) \subset L^2(\Gamma) \to L^2(\Gamma)
\end{equation*}
has domain $\Dom$ contained in 
\begin{equation*}
\Dom_{\max}(A) = \set{u \in L^2(\Gamma)\st Au \in L^2(\Gamma)}
\end{equation*}
and containing $\Dom_{\min}(A)$, the closure of $C_c^\infty(\open
\Gamma)\subset \Dom_{\max}(A)$ with respect to the graph norm
\begin{equation*}
\|u\|_A = \|u\|_{L^2(\Gamma)} + \|Au\|_{L^2(\Gamma)}.
\end{equation*}
The operator $A$ is Fredholm with either of these canonical domains and therefore so is $A_\Dom$, the operator
\begin{equation}\label{AwithDomain}
A:\Dom\subset L^2(\Gamma)\to L^2(\Gamma),
\end{equation}
whenever $\Dom$ is a subspace of $\Dom_{\max}(A)$ containing
$\Dom_{\min}(A)$. In particular, the quotient
$\Dom_{\max}(A)/\Dom_{\min}(A)$ is a finite-dimensional space. Lesch's
book \cite{Le97} is a systematic account of these and many other aspects
of general elliptic cone operators, including the fact that such
operators, on compact manifolds with conical singularities, are Fredholm
operators with any of the domains $\Dom$. For ordinary differential
equations these facts can of course be verified directly.

Choosing a domain amounts to selecting a subspace of
$\Dom_{\max}(A)/\Dom_{\min}(A)$, which (as is well known) amounts
essentially to prescribing linear homogeneous relations between the values
and ``first derivatives'' of elements $u\in \Dom_{\max}(A)$ at $\partial G$.
Such a choice is in principle completely arbitrary,
and may have nothing to do with the graph in question. There is, however,
a subclass of domains that are, in a natural sense, compatible with, or
which respect, the original graph. Among these we will single out a
smaller class of domains that are specified by what we shall call 
admissible coupling conditions, fully described in Section~\ref{s-DomainsAndNetworks}. Aside from being
compatible with the graph, coupling conditions ensure that the
operator \eqref{AwithDomain} has index zero (which is necessary for
an operator to have a nonempty resolvent set).

It is easy to see that in order for $A_\Dom$ to have $\Lambda$ as a sector
of minimal growth, none of the coefficients $a_j$ in \eqref{ALocalEdge}
should have values in $\Lambda$. But this condition is in general not
sufficient. One needs to also analyze whether $\Lambda$ is a sector of
minimal growth for the so called model operator $A_{p,\wedge}$ of $A$ at
a vertex $p$ with a specified domain determined by $\Dom$.
Section~\ref{Sec-ModelOperator} concerns the precise definition of
$A_{p,\wedge}$ and an analysis of its spectrum,
while Section~\ref{Sec-DecModelOperator} concerns sectors of minimal
growth. Thus these two sections form the core of this paper; 
the arguments rely on \cite{GKM1,GKM3}.

The results of Section~\ref{Sec-DecModelOperator} are assembled to give
Theorem~\ref{SectorMinGrowthGraph}, the main result of this paper.
Finally, in Section~\ref{Sec-ConeOperators} we indicate how the analysis
can be extended to the case where the operators $A_j$ have arbitrary order
and regular singular points at most at the endpoints of the edges.

%%%%%%%%%%%%%%%%%%%%%%%%%%%%%%%%%%%%%%%%%%%%%%%%%%%%%%%%%%%
\section{Domains and graphs}\label{s-DomainsAndNetworks}

Since $A$ is elliptic and the principal part is regular and smooth up to the boundary, we have $\Dom_{\min}(A) = H_0^2(G)$, the space of functions that are $H^2$-regular on $G$ and vanish to second order at the vertices.

The functions in $\Dom_{\max}(A)$ also exhibit $H^2$ regularity in the interior. However, since $A$ is singular, these functions are not regular up to the boundary. To describe their behavior near $\partial G$, choose once and for all a smooth defining function $x$ for $G$, i.e., $x\geq 0$, and $x=0$ precisely on $\partial G$ with $dx\ne 0$ at each $q\in \partial G$. For instance, we can choose $x=1-s^2$ on $E_j \cong [-1,1]$. Using $x$
as coordinate near an endpoint $q\in \partial G$ on the interval $E_j$
with $q\in \partial E_j$ (so $x \in [0,\eps)$, $\eps>0$ small, and $x=0$
at $q$), we have
\begin{equation}\label{AtAnEndpoint}
A_j = a_q D_x^2 + b_q D_x + \frac{c_q }{x}
\end{equation}
for some functions $a_q$, $b_q$, $c_q$, smooth near $q$, $a_q\ne 0$ at
$q$. Modulo $H_0^2([0,\eps))$ the generic asymptotic behavior of an
arbitrary $u \in \Dom_{\max}(A)$ near $q$ is
\begin{equation}\label{AsympAnsatz}
u(x)|_{E_j} \sim \alpha_q + \gamma_q x\log x + \beta_q x \quad
\text{as }x \to 0,
\end{equation}
in principle with arbitrary constants $\alpha_q$, $\beta_q$, $\gamma_q\in
\C$; this follows from analyzing the indicial polynomial of $A_j$ at $q$, see for example \cite{BK}.
Using \eqref{AtAnEndpoint} we get modulo $L^2$
\begin{align*}
A_ju(x) &\sim \Big(a_q(x) D_x^2 + b_q(x)D_x + \frac{c_q(x)}{x}\Big)
      (\alpha_q + \gamma_q x\log(x) + \beta_q x) \\
 &= -a_q(x)\frac{\gamma_q}{x}+\frac{1}{i}b_q(x)
    \big(\gamma_q\log x+\gamma_q+\beta_q\big) + \frac{c_q(x)}{x}
    \big(\alpha_q+\gamma_qx\log x+\beta_q x\big) \\
 &\sim \frac{-a_q(0)\gamma_q+c_q(0)\alpha_q}{x}
\end{align*}
as $x\to 0$ for every $j=1,\dots,k_p$.  Thus, $Au(x)|_{E_j}$ is in $L^2$ as $x\to 0$ if and only if
\begin{equation*}
\gamma_q= \frac{c_q(0)}{a_q(0)}\alpha_q.
\end{equation*}
Let $\omega_q\in C_c^\infty(E_j)$ be equal to $1$ near $q$ and equal to $0$ near the other endpoint of $E_j$. Then \begin{equation}\label{GenericUinEq}
u|_{E_j}=\omega_q \big(\alpha_q(1+\frac{c_q(0)}{a_q(0)}x\log x)+\beta_q
x\big)+v_q
\end{equation}
where $v_q\in H^2_0([0,\eps))$ near $q$. Summing up, we have:

\begin{lemma}\label{DefEsubq}
For each $q\in \partial G$ let $\Sing_q(A)$ be the space of functions on
$\R_+$ of the form
\begin{equation}\label{aSingFunction}
\alpha_q(1+\frac{c_q(0)}{a_q(0)}x\log x)+\beta_q x,
\end{equation}
and let $\omega_q\Sing_q(A)$ mean the space obtained by multiplying the
elements of $\Sing_q(A)$ by $\omega_q$ and regarded as functions on $G$
supported on $E_j$.
Then
\begin{equation}\label{aSingFunction11}
\Dom_{\max}(A)=\Dom_{\min}(A)+ \bigoplus_{q\in \partial G}
\omega_q\Sing_q(A),
\end{equation}
and so
\begin{equation}\label{EndPointDecomposition}
\Dom_{\max}(A)/\Dom_{\min}(A) \cong \bigoplus_{q\in \partial G} \Sing_q(A)
\end{equation}
canonically.
\end{lemma}

We note in particular that there is a one-to-one onto correspondence between domains and subspaces of $\Dom_{\max}(A)/\Dom_{\min}(A)$, so with subspaces of
\begin{equation*}
\Sing(A)=\bigoplus_{q\in \partial G} \Sing_q(A).
\end{equation*}
For this reason we will often refer to domains simply as subspaces of $\Sing(A)$, and view them as elements of the various Grassmannians $\Gr_\ell(\Sing(A))$ of $\ell$-dimensional subspaces of $\Sing(A)$.

If $p\subset \partial G$, let
\begin{equation}\label{VertexDecomposition2}
\Sing_p(A)=\bigoplus_{q\in p} \Sing_q(A),
\end{equation}
and let
\begin{equation*}
\pi_p:
\Sing(A)\to \Sing_p(A)
\end{equation*}
be the canonical projection.
\begin{definition}
Let $\sim$ be a relation of equivalence on $\partial G$, let $V=G/{\sim}$
and let $\Gamma$ be the corresponding graph. We shall say that a domain
$\Dom$ and a graph $\Gamma$ are compatible if
\begin{equation*}
\Dom/\Dom_{\min}(A) = \bigoplus_{p\in V} \pi_p(\Dom/\Dom_{\min}(A)).
\end{equation*}
\end{definition}
For example, every domain is compatible with the graph with single vertex $\partial G$. The domain generated by the elements \eqref{aSingFunction} with $\alpha_q=0$ for all $q$ (the Dirichlet domain) is compatible with $G$, as are $\Dom_{\min}(A)$ and $\Dom_{\max}(A)$.

For any domain $\Dom$ there is a compatible graph with maximal number of vertices. The domains that are compatible with a given a graph $\Gamma$ with vertices $V$ are simply those of the form
\begin{equation}\label{VertexDecomposition1}
\Dom/\Dom_{\min}(A)=\bigoplus_{p\in V} \Dom_p,\ \Dom_p \subset \Sing_p(A).
\end{equation}
Denote by $R\subset \partial G\times \partial G$ the relation of
equivalence determining the vertices of $\Gamma$. Any other relation of
equivalence containing $R$ will give a graph also compatible with $\Dom$.
The set of relations of equivalence giving graphs compatible with
$\Dom$ forms a partially ordered set (ordered by inclusion) with a unique
minimal element. The graph associated with that minimal relation of
equivalence might be called the graph determined by $\Dom$.     

\medskip
Since $G$ has $N$ edges, \eqref{EndPointDecomposition} gives
\begin{equation} \label{MaxMinDimension}
\dim \Dom_{\max}(A)/\Dom_{\min}(A) = 4N.
\end{equation}

\begin{proposition}\label{InjectiveSurjective}
$A:\Dom_{\min}(A)\to L^2(\Gamma)$ is injective and $A:\Dom_{\max}(A)\to L^2(\Gamma)$
is surjective.
\end{proposition} 
\begin{proof}
Let $A_{\min}=A|_{\Dom_{\min}(A)}$ and $A_{\max}=A|_{\Dom_{\max}(A)}$. If
$A^\star$ denotes the formal adjoint of $A$, then the $L^2$-adjoint
$A_{\max}^*$ of $A_{\max}$ is precisely $A^\star_{\min}$. Thus, since $\rg
(A_{\max})$ is orthogonal to $\ker(A_{\max}^*)$, and since $A$ and
$A^\star$ are of the same form, it is enough to check the injectivity of
$A_{\min}$.

Observe that if we consider $A$ acting on distributions over
$\open \Gamma$, then its kernel satisfies the relation
\begin{equation*}
\ker(A) \cong \bigoplus_{j=1}^N \ker (A_j),
\end{equation*}
where $A_j$ denotes the restriction of $A$ to $E_j$ acting on
distributions over $\open E_j$.

Now, for each operator $A_j$ we have $\dim \ker (A_j)=2$, so $\dim
(L^2\cap\ker A_j)\leq 2$.  Consequently, for every domain
$\Dom_{\min}\subset \Dom \subset
\Dom_{\max}$,
\begin{equation*}
\dim\ker(A_{\Dom})\le 2N \;\text{ and }\; 
   \dim\ker(A_{\Dom}^*)\le 2N
\end{equation*}
since $A_{\Dom}^*$ is a closed extension of $A^\star$. Therefore,
\begin{equation*}
-2N\le \Ind A_{\Dom} = \dim\ker(A_{\Dom}) - \dim\ker(A_{\Dom}^*) \le 2N.
\end{equation*}
In particular,
\begin{equation*}
-2N\le \Ind A_{\min} \quad\text{and}\quad \Ind A_{\max}\le 2N.
\end{equation*}
On the other hand,
\begin{align*}
\Ind A_{\max} &= \Ind A_{\min} + \dim \Dom_{\max}/\Dom_{\min} \\
 &= \Ind A_{\min} + 4N 
\end{align*}
by \eqref{MaxMinDimension}. Thus $\Ind A_{\min} + 4N\le 2N$ and so $\Ind
A_{\min} \le -2N$. Hence $\Ind A_{\min} = -2N$.  Finally, since $\dim \ker
A_{\min}^*\le 2N$, we get
\begin{equation*}
-2N = \dim \ker A_{\min} - \dim \ker A_{\min}^* \ge \dim \ker A_{\min} -
2N,
\end{equation*}
which implies $\dim \ker A_{\min}=0$.
\end{proof}

As a consequence of \eqref{MaxMinDimension} and Proposition~\ref{InjectiveSurjective} we see that the domains $\Dom$ for which \eqref{AwithDomain} has index $0$ satisfy $\dim\Dom/\Dom_{\min}(A)=2N$. Such a domain can be specified as the kernel of a surjective linear map 
\begin{equation}\label{DomainAsKernel}
\gamma:\Sing(A)\to\C^{2N}.
\end{equation}
We will write
\begin{equation*}
\Dom_\gamma = \set{u \in \Dom_{\max}(A)\st \gamma(u/\Dom_{\min}(A)) = 0}
\end{equation*}
for the domain associated with such $\gamma$.
Of course, composing $\gamma$ with an isomorphism of $\C^{2N}$ gives a map
$\tilde\gamma$ with the same kernel as $\gamma$; we will regard
$\gamma$ and $\tilde\gamma$ as equivalent. 

Fix a graph $\Gamma=G/{\sim}$. The domains that are compatible with $\Gamma$ are of the form \eqref{VertexDecomposition1}. Among these we single out those given as follows. Let $k_p=\#p$ be the cardinality of $p$, let
\begin{equation}\label{CpSingp}
\gamma_p:\Sing_p(A)\to \C^{k_p}
\end{equation}
be linear and surjective, and let 
\begin{equation}\label{AssembleDomainAsKernel}
\gamma:\bigoplus_{p\in V}\Sing_p(A)\to \bigoplus_{p\in V}\C^{k_p}
\end{equation}
be the obvious block-diagonal map determined by the $\gamma_p$. Since
$\sum_{p\in V} k_p=2N$, this is a surjective map \eqref{DomainAsKernel}.

\begin{definition}
An admissible coupling condition is a condition $\gamma u=0$, where
$\gamma$ is a block-diagonal map \eqref{AssembleDomainAsKernel} for which the restrictions \eqref{CpSingp} are surjective. 
\end{definition}
Two admissible coupling conditions are regarded as equivalent if one can
be obtained from the other by composition on the left with an invertible
(block-diagonal) operator.  The set of these equivalence classes is in 
one-to-one correspondence with the submanifold 
\begin{equation*}
\prod_{p\in V}\Gr_{k_p}(\Sing_p(A))
\end{equation*}
of $\Gr_{2N}(\Sing(A))$, the Grassmannian of domains of index $0$ for $A$.

\medskip
Fixing a defining function $x$, as we have done, fixes a basis of $\Sing(A)$, so we may express each $\gamma_p$ (and of course also $\gamma$) as a matrix. Let $\set{q_1,\dotsc,q_{k_p}}$ be an enumeration of the elements of $p$. Then
\begin{equation}\label{MatrixVertexGlobal}
\gamma_p : \sum_{j=1}^{k_p} \big(
 \alpha_{q_j} (1 + \tfrac{c_{q_j}(0)}{a_{q_j}(0)}x\log x)+ \beta_{q_j}
x\big)
\mapsto C_p\alpha + C_p'\beta
\end{equation}
where $\alpha$ is the column with entries $\alpha_{q_j}$, similarly
$\beta$,
and
\begin{equation*}
C_p=\begin{pmatrix}
c_{1,1} & \cdots & c_{1,k_p} \\
\vdots & \ddots & \vdots \\
c_{k_p,1} & \cdots & c_{k_p,k_p}
\end{pmatrix},\quad 
C_p'=\begin{pmatrix}
c_{1,1}' & \cdots & c_{1,k_p}' \\
\vdots & \ddots & \vdots \\
c_{k_p,1}' & \cdots & c_{k_p,k_p}'
\end{pmatrix}.
\end{equation*}
We shall identify the linear map $\gamma_p$ with the $(k_p \times
2k_p)$-matrix
\begin{equation}\label{MatrixVertex}
\gamma_p=
\begin{pmatrix}
C_p & C_p'
\end{pmatrix}
\end{equation}
and regard two such matrices as equivalent if one is obtained from 
the other by multiplication on the left by an invertible 
$(k_p\times k_p)$-matrix. For any positive integer $k$ we let 
$V_{k,2k}(\C)$ be the set of $k\times 2k$ complex matrices with maximal rank, and write $G_{k,2k}(\C)$ for the quotient of $V_{k,2k}(\C)$ by the standard action of $\mathrm{GL}(k,\C)$ on the left. Thus fixing a defining function $x$ for $\partial G$ and the bases of the various $\Sing_q(A)$ as indicated above, establishes specific isomorphisms $\Gr_{k_p}(\Sing_{p}(A))\to G_{k_p,2k_p}(\C)$.

%%%%%%%%%%%%%%%%%%%%%%%%%%%%%%%%%%%%%%%%%%%%%%%%%%%%%%%%%%%%%%%%%%%%%%
\section{The model operator}\label{Sec-ModelOperator}
We continue our discussion with a fixed defining function $x$ for
$\partial G$. With respect to this function, the operator $A$ has the form
\eqref{AtAnEndpoint} at $q\in \partial G$.

\begin{definition}
The model operator of $A$ at $p\in V$ is defined to be the diagonal
operator 
\begin{equation}\label{ApwedgeMatrix}
A_{p,\wedge}=\bigoplus_{q\in p} a_q(0) D_x^2:\bigoplus_{q\in p}
C_c^\infty(\R_+)\to \bigoplus_{q\in p}
C_c^\infty(\R_+).
\end{equation}
Since $A$ is elliptic we have $a_q(0) \neq 0$.
\end{definition}

The canonical domains of $A_{p,\wedge}$ are 
\begin{equation*}
\Dom_{\min}(A_{p,\wedge}) = \bigoplus_{q\in p} H^2_0(\overline\R_+),
\;\quad\;
\Dom_{\max}(A_{p,\wedge}) = \bigoplus_{q\in p} H^2(\R_+),
\end{equation*}
and, as in Section~\ref{s-DomainsAndNetworks},
\begin{equation*}
\Dom_{\max}(A_{p,\wedge})/\Dom_{\min}(A_{p,\wedge}) \cong
\Sing_p(A_\wedge)
\end{equation*}
where
\begin{equation*}
\Sing_p(A_{\wedge})=\bigoplus_{q\in p}\set{\alpha_q + \beta_q
x:\alpha_q,\ \beta_q\in \C}.
\end{equation*}
This space and the space $\Sing_p(A)$ in
Section~\ref{s-DomainsAndNetworks} are linked by the map 
\begin{equation}\label{thetap}
\begin{gathered}
\theta_p: \Sing_p(A) \to \Sing_p(A_{\wedge}), \\ 
\theta_p \Big(\bigoplus_{q\in p}
\big(\alpha_q(1+\tfrac{c_q(0)}{a_q(0)}x\log x) + \beta_q x\big)\Big) 
=\bigoplus_{q\in p}(\alpha_q + \beta_q x).
\end{gathered}
\end{equation}
This isomorphism, which allows us to relate domains for $A$ with domains
for $A_{p,\wedge}$, will be a key component in the proof of 
Theorem~\ref{SectorMinGrowthGraph}. 

An admissible domain for $A_{p,\wedge}$ is a subspace of 
$\Dom_{\max}(A_{p,\wedge})$ given by 
\begin{equation*}
\Dom_{\gamma_p}(A_{p,\wedge}) =
\set{u\in
\Dom_{\max}(A_{p,\wedge}):\gamma_p(u/\Dom_{\min}(A_{p,\wedge}))=0},
\end{equation*}
where $\gamma_p:\Sing_p(A_{\wedge})\to \C^{k_p}$, $k_p=\#p$, is a 
surjective map.  

Let
\begin{equation}\label{bgreswedge}
\begin{aligned}
\bgres(A_{p,\wedge}) = \bigl\{\lambda \in \C \st A_{p,\wedge}-\lambda\; 
&\textup{ is injective on } \Dom_{\min}(A_{p,\wedge})\\
&\textup{ and surjective on } \Dom_{\max}(A_{p,\wedge})\bigr\},
\end{aligned}
\end{equation}
the \emph{background resolvent set} of $A_{p,\wedge}$, and let the
\emph{background spectrum} $\bgspec(A_{p,\wedge})$ be the complement
of $\bgres(A_{p,\wedge})$ in $\C$. The relevance of the background
spectrum lies in the fact that
$$
\bgspec(A_{p,\wedge}) = \bigcap_{\Dom_{\min} \subset \Dom \subset
\Dom_{\max}}\spec\bigl(A_{p,\wedge}\big|_{\Dom}\bigr),
$$
i.e., it is the part of the spectrum common to all extensions.

\begin{lemma}\label{bgspecwedge}
We have
$$
\bgspec(A_{p,\wedge}) = \bigcup_{q\in p} a_q(0)\cdot\overline{\R}_+.
$$
\end{lemma}
\begin{proof}
Without loss of generality we may assume $k_p=1$ and $a_q(0)=1$, thus 
$A_{p,\wedge}=D_x^2$. The extension of $A_{p,\wedge}$ with 
domain $H^2(\R_+)\cap H^1_0(\overline{\R}_+)$ (the Dirichlet extension)
is selfadjoint and 
positive, so it is clear that the background spectrum of $A_{p,\wedge}$ 
is contained in $\overline{\R}_+$. On the other hand, we also have 
$\overline{\R}_+\subset\bgspec(A_{p,\wedge})$ since the operator
$$
A_{p,\wedge} - \lambda : \Dom_{\max}(A_{p,\wedge}) \to L^2(\R_+)
$$
is not surjective for every $\lambda \geq 0$.
To see this write $\lambda = \mu^2$ with $\mu \geq 0$, and
$$
A_{p,\wedge} - \lambda = \bigl(D_x + \mu\bigr)\bigl(D_x - \mu\bigr).
$$
If $A_{p,\wedge} - \lambda$ was surjective, then also $D_x + \mu =
e^{-ix\mu}D_xe^{ix\mu}: H^1(\R_+) \to L^2(\R_+)$
would be surjective. Consequently, also $D_x$ would be surjective
contradicting the fact that any function $\varphi \in
L^2(\R_+)$ with $\varphi(x) = i/x$ for large values of $x$ does not have an
antiderivative in $L^2$.
\end{proof}

For $\lambda \in \bgres(A_{p,\wedge})$ we let
$$
\K_{p,\wedge}(\lambda) = 
\ker\bigl((A_{p,\wedge}-\lambda)\big|_{\Dom_{\max}}\bigr) =
\bigoplus_{q\in p}\Bigl\{\alpha_q\,e^{-\sqrt{-\lambda/a_q(0)}\;x}\st
\alpha_q \in \C\Bigr\},
$$
where the roots $\sqrt{-\lambda/a_q(0)}$ are chosen such that 
\begin{equation}\label{RootsChoice} 
  \Re \Big(\sqrt{-\lambda/a_q(0)}\Big)>0. 
\end{equation}
Modulo $H^2_0([0,\eps))$, we have
$$
\alpha_q\,e^{-\sqrt{-\lambda/a_q(0)}\;x} \sim 
\alpha_q\,\Bigl(1 - \sqrt{-\lambda/a_q(0)}\;x\Bigr)
\;\text{ as } x \to 0,
$$
so
\begin{multline}\label{KernelAsymptotics}
\bigl(\K_{p,\wedge}(\lambda) + \Dom_{\min}(A_{p,\wedge})\bigr)
/\Dom_{\min}(A_{p,\wedge}) \\
\cong \bigoplus_{q\in p}\Bigl\{\alpha_q\,
\Bigl(1 - \sqrt{-\lambda/a_q(0)}\;x\Bigr)\st \alpha_q \in\C\Bigr\}.
\end{multline}

We need the following simple algebraic lemma.
\begin{lemma}\label{AlgebraicIntermediate}
Let $\Dom_{\min}$, $\Dom_{\max}$, and $H$ be vector spaces, 
$\Dom_{\min} \subset \Dom_{\max}$. Let
$\mathcal A:\Dom_{\max}\to H$ be a surjective linear map that is 
injective on $\Dom_{\min}$. 
Let $\Dom$ be any intermediate space $\Dom_{\min}\subset \Dom \subset 
\Dom_{\max}$.  Then $\mathcal A: \Dom \to H$ is bijective if and only 
if $\Dom_{\max}/\Dom_{\min} = \Dom/\Dom_{\min} \oplus 
\bigl(K+\Dom_{\min}\bigr)/\Dom_{\min}$, where $K$ is the
kernel of $\mathcal A$ on $\Dom_{\max}$.
\end{lemma}

Let $\Dom_{\gamma_p}(A_{p,\wedge})$ be an admissible domain. 
Using the lemma with $\mathcal A=A_{p,\wedge}-\lambda$,
and since 
\[ \dim \Dom_{\gamma_p}(A_{p,\wedge})/\Dom_{\min}(A_{p,\wedge})=
   \dim \bigl(\K_{p,\wedge}(\lambda) + \Dom_{\min}(A_{p,\wedge})\bigr)
   /\Dom_{\min}(A_{p,\wedge}) = k_p, \]
we conclude that 
\begin{gather} \notag
\lambda\in \spec\bigl(A_{p,\wedge}\big|_{\Dom_{\gamma_p}}\bigr)
   \;\text{ if and only if} \\ \label{IntersecCond}
\Dom_{\gamma_p}(A_{p,\wedge})/\Dom_{\min}(A_{p,\wedge}) \cap
\bigl(\K_{p,\wedge}(\lambda) + \Dom_{\min}(A_{p,\wedge})\bigr)/
\Dom_{\min}(A_{p,\wedge}) \neq \{0\}.
\end{gather}
Note that $\dim \Dom_{\max}(A_{p,\wedge})/\Dom_{\min}(A_{p,\wedge})=2k_p$.

Let $p=\set{q_1,\dotsc,q_{k_p}}$. 
As in Section~\ref{s-DomainsAndNetworks} we shall identify $\gamma_p$ with 
a matrix $(C_p\;\, C_p') \in V_{k_p,2k_p}(\C)$, and specifying a domain 
$\Dom_{\gamma_p}(A_{p,\wedge})$ is equivalent to specifying an equivalence 
class $[\gamma_p] \in G_{k_p,2k_p}(\C)$.

To simplify the notation we will write $a_j(0)$ instead of $a_{q_j}(0)$. 

\begin{proposition}\label{SpectrumonWedge}
Let $\lambda \in \bgres(A_{p,\wedge})$ and let
$\gamma_p = (C_p \;\, C_p')$ be an admissible 
coupling condition at $p$.  Let $\Delta(\lambda)$ be the diagonal 
matrix with entries $\sqrt{-\lambda/a_j(0)}$, $j=1,\dots,k_p$, 
chosen as in \eqref{RootsChoice}. Then 
\begin{equation}\label{AwedgeDet}
\lambda \in \spec\bigl(A_{p,\wedge}\big|_{\Dom_{\gamma_p}}\bigr)
\;\text{ if and only if }\;\;
\det \bigl(C_p-C_p'\Delta(\lambda)\bigr)=0.
\end{equation}
\end{proposition}
In other words, $\spec\bigl(A_{p,\wedge}\big|_{\Dom_{\gamma_p}}\bigr)$ 
consists of the background spectrum and the 
solutions of the determinant equation in \eqref{AwedgeDet}.
\begin{proof}
Let $u=\oplus_{j=1}^{k_p}\alpha_j\big(1 - \sqrt{-\lambda/a_j(0)}\;x\big) 
\in \bigl(\K_{p,\wedge}(\lambda) + \Dom_{\min}(A_{p,\wedge})\bigr)/
\Dom_{\min}(A_{p,\wedge})$. Then $u \in \Dom_{\gamma_p}(A_{p,\wedge})/
\Dom_{\min}(A_{p,\wedge})$ if and only if
\begin{equation*}
\begin{pmatrix} C_p & C_p' \end{pmatrix}
\begin{pmatrix} I \\ -\Delta(\lambda) \end{pmatrix} \alpha =
\bigl(C_p-C_p'\Delta(\lambda)\bigr)\alpha = 0,
\end{equation*}
where $\alpha$ is the column vector with entries $\alpha_j$.
Therefore, \eqref{IntersecCond} is satisfied if and only if
$C_p-C_p'\Delta(\lambda)$ is not invertible.
\end{proof}

\begin{example}[$\delta$-type conditions, cf. \cite{Kuchment2}]
\label{deltaConditions}
Consider the condition given by the $(k_p\times2k_p)$-matrix (consisting 
of two square blocks)
\begin{equation}\label{deltaConditionsMatrix}
\gamma_p = \left(
\begin{array}{crrrrc|cccccc}
1 &-1 & 0 & \cdots & 0 & 0 & 0 & 0 & 0 &\cdots & 0 &0\\
0 & 1 &-1 & & 0 & 0 & 0 & 0 & 0 &\cdots & 0 &0\\[1ex]
\vdots & & \hspace*{1ex}\ddots \hspace*{-1ex} & 
\hspace*{2ex}\ddots \hspace*{-2ex}& 
& \vdots & \vdots & \vdots & \vdots & & \vdots & \vdots\\[2ex]
0 & 0 & 0 & & 1 & -1\; & 0 & 0 & 0 &\cdots & 0 &0\\
\hspace*{-1ex} \nu_p \hspace*{-1ex}& 0 & 0 & \cdots & 0 & 0 
& \;\; c_1' & c_2' & c_3' &\cdots & 
c_{k_p-1}'\hspace*{-1ex} & c_{k_p}' \hspace*{-1ex}
\end{array}
\right),
\end{equation}
where $(\nu_p,c_1',\dots,c_{k_p}')\in \C^{k_p+1}\minus\{0\}$.  
Let $A_{p,\wedge}$ be as in \eqref{ApwedgeMatrix}.
By Proposition~\ref{SpectrumonWedge}, $\lambda\in\bgres(A_{p,\wedge})$
belongs to the spectrum of $A_{p,\wedge}$ with domain $\Dom_{\gamma_p}$ 
if and only if 
\begin{equation*}
\det\left(
\begin{array}{ccccc}
1 & \hspace*{-1.5ex} -1 & 0 & \cdots & 0 \\
0 & 1 & \hspace*{-1.5ex} -1 & \cdots & 0 \\
0 & 0 &  1 & & 0 \\
\vdots & \vdots & & \ddots & \\
0 & 0 & 0 & & \hspace*{-1.5ex} -1 \\
\hspace*{-1ex} \nu_p - c_1'\sqrt{\frac{-\lambda}{a_1(0)}} 
& -c_2'\sqrt{\frac{-\lambda}{a_2(0)}} & -c_3'\sqrt{\frac{-\lambda}{a_3(0)}}
& \cdots & -c_{k_p}'\sqrt{\frac{-\lambda}{a_{k_p}(0)}} 
\end{array}
\right)
=0,
\end{equation*}
which is equivalent to
\begin{equation}\label{EigenwertEquation}
\sum_{j=1}^{k_p} c_j' \sqrt{-\lambda/a_j(0)} = \nu_p. 
\end{equation}

In the case of Kirchhoff boundary conditions $\nu_p=0$ and
$c_1'=\cdots=c_{k_p}'=1$, the equation \eqref{EigenwertEquation} 
has no solution and we get
\[ \spec(A_{p,\wedge}\big|_{\Dom_{\gamma_p}}) = \bgspec(A_{p,\wedge})
  =\bigcup_{j=1}^{k_p} a_j(0)\cdot\overline{\R}_+. \]

Now let $c_1',\dots,c_{k_p}'$ be arbitrary complex numbers.
Given an open sector $\Lambda_0$ in $\bgres(A_{p,\wedge})$ we let
$w:-\Lambda_0\to\C$ be a holomorphic square root and let the
roots $\sqrt{a_j(0)}$ be chosen such that 
\[ \sqrt{-\lambda/a_j(0)}=\frac{w(-\lambda)}{\sqrt{a_j(0)}}
   \;\;\text{ for } \lambda\in\Lambda_0, \] 
where $\sqrt{-\lambda/a_j(0)}$ is as above the square root with 
positive real part.  Then, over $\Lambda_0$, the equation 
\eqref{EigenwertEquation} can be written as 
\[ w(-\lambda)\sum_{j=1}^{k_p}\frac{c_j'}{\sqrt{a_j(0)}} = \nu_p, \]
and there are three possible outcomes: 
\begin{enumerate}[\quad $(a)$]
\item The equation has no solution, which implies
$$
\Lambda_0 \cap \spec\bigl(A_{p,\wedge}\big|_{\Dom_{\gamma_p}}\bigr) 
= \varnothing.
$$
\item The equation has the unique solution 
$$
\lambda_p = -\nu^2_p/ \Bigl(\sum_{j=1}^{k_p}
\frac{c_j'}{\sqrt{a_j(0)}}\Bigr)^2,
$$
in which case
$$
\Lambda_0 \cap \spec\bigl(A_{p,\wedge}\big|_{\Dom_{\gamma_p}}\bigr)
= \{\lambda_p\}.
$$
\item Every $\lambda\in\Lambda_0$ solves the equation, and so
$\Lambda_0 \subset \spec\bigl(A_{p,\wedge}\big|_{\Dom_{\gamma_p}}\bigr)$.
Note that this can occur only for $\nu_p = 0$.
\end{enumerate}
\end{example}

Clearly, the choice of $\nu_p, c_1',\dots,c_{k_p}'$ in the coupling 
condition $\gamma_p$ determines the spectral behavior of $A_{p,\wedge}$ 
with domain $\Dom_{\gamma_p}$.  In order to illustrate the subtlety of 
the spectrum (even in this simple example), let's consider the special
case $\nu_p=0$.   

Note that if $A_{p,\wedge}$ has $n_p$ coefficients with distinct 
arguments, then the background resolvent set is a disjoint union of 
$n_p$ open sectors 
\begin{equation}\label{bgresSplitting}
 \bgres(A_{p,\wedge}) = \bigcup_{j=1}^{n_p} \Lambda_j.
\end{equation}
By the above discussion, if $\nu_p=0$, then for each $j$
\[
\text{either }\; \Lambda_j \cap 
\spec\bigl(A_{p,\wedge}\big|_{\Dom_{\gamma_p}}\bigr) = \varnothing
\;\;\text{ or }\;\;
\Lambda_j \subset \spec\bigl(A_{p,\wedge}\big|_{\Dom_{\gamma_p}}\bigr).
\]
In fact, depending on the choice of $c_1',\dots,c_{k_p}'$, any 
collection of these sectors may or may not be in the spectrum of 
$A_{p,\wedge}$ with domain $\Dom_{\gamma_p}$.

Let $a_j^0=e^{i\varphi_j}$, $j=1,\dots,n_p$, be an enumeration of 
the elements of the set
$\big\{a_\ell(0)/|a_\ell(0)|\st \ell=1,\dots,k_p\big\}$, ordered 
in such a way that $0\le \varphi_1<\cdots<\varphi_{n_p}<2\pi$. Let 
\begin{equation}\label{bgresSectors}
\begin{gathered}
\Lambda_j= \big\{re^{i\varphi}\st r>0 \text{ and }
  \varphi_j<\varphi<\varphi_{j+1}\big\},\;\; 1\le j \le n_p-1, \\
\Lambda_{n_p}= \big\{re^{i\varphi}\st r>0 \text{ and }
  \varphi_{n_p}<\varphi<\varphi_1+2\pi\big\}.
\end{gathered}
\end{equation}
These are the components of $\bgres(A_{p,\wedge})$. 
For $\lambda\in\bgres(A_{p,\wedge})$ we choose $\sqrt{-\lambda/a_j^0}$ 
such that $\Re\big(\sqrt{-\lambda/a_j^0}\big)>0$ for all $j$.
Fix $\sqrt{a_1^0}$ arbitrarily.  The function $\lambda\mapsto 
\sqrt{a_1^0}\sqrt{-\lambda/a_1^0}$ is holomorphic on 
$\C\minus(a_1^0\cdot \overline{\R}_+)$.  Choose
$\sqrt{a_2^0},\dots,\sqrt{a_{n_p}^0}$ such that
\begin{equation*}
\sqrt{-\lambda/a_j^0} = \frac{\sqrt{a_1^0}\sqrt{-\lambda/a_1^0}}
{\sqrt{a_j^0}}\quad \text{on }\Lambda_j.
\end{equation*} 
Let $\Omega_{1,j}=\C\minus\big[(a_1^0\cdot \overline{\R}_+) \cup 
(a_j^0\cdot \overline{\R}_+)\big]$ and let 
$\epsilon_j:\Omega_{1,j} \to \{\pm1\}$ be defined by
\begin{equation}\label{SignFunction}  
 \epsilon_j(\lambda)= \frac{\sqrt{a_j^0}\sqrt{-\lambda/a_j^0}}
 {\sqrt{a_1^0}\sqrt{-\lambda/a_1^0}}. 
\end{equation}
Note that $\epsilon_j(\lambda)=1$ for $\lambda\in\Lambda_j$.

\begin{lemma}\label{EpsilonValues}
The functions $\epsilon_j$ satisfy 
\begin{equation*}
\epsilon_j(\Lambda_\ell)=
\begin{cases}
-1 &\text{if } 1\le\ell< j, \\
\hspace*{1.5ex} 1 &\text{if } j\le\ell\le n_p.
\end{cases}
\end{equation*} 
\end{lemma}
\begin{proof}
Clearly, $\epsilon_1(\lambda)=1$ for every $\lambda\in \Omega_{1,1}$. 
If $j\not=1$, then $\Omega_{1,j}$ has
two connected components. Since $\sqrt{-\lambda/a_j^0}$ is continuous 
across the ray $a_1^0\cdot \overline{\R}_+$, and since for
$\lambda_{\varphi}=e^{i\varphi}a_1^0$ we have
\[ 
   \lim_{\varphi\to 0^-}\sqrt{-\lambda_{\varphi}/a_1^0}=i
   \quad\text{and}\quad 
   \lim_{\varphi\to 0^+}\sqrt{-\lambda_{\varphi}/a_1^0}=-i,
\]
we see that $\epsilon_j$ changes sign across the ray 
$a_1^0\cdot \overline{\R}_+$,  so $\epsilon_j$ is $1$ in the component 
of $\Omega_{1,j}$ containing $\Lambda_j$ and $-1$ in the other. Thus 
$\epsilon_j=-1$ on $\bigcup_{\ell=1}^{j-1}\Lambda_\ell$, and
$\epsilon_j=1$ on $\bigcup_{\ell=j}^{n_p}\Lambda_\ell$. 
\end{proof}

\begin{proposition}\label{SectorsSpectrum}
Let $\gamma_p$ be given by \eqref{deltaConditionsMatrix} with
$\nu_p=0$. Then, for any collection of $m<k_p$ components of 
$\bgres(A_{p,\wedge})$, there is a choice of $(c_1',\dots,c_{k_p}')
\in \C^{k_p}\minus\{0\}$ in $\gamma_p$ such that
\[ \spec\bigl(A_{p,\wedge}\big|_{\Dom_{\gamma_p}}\bigr)
   = \bgspec(A_{p,\wedge})\cup \bigcup_{k=1}^m \Lambda_{j_k}. \]
\end{proposition}
\begin{proof}
Let $\Lambda_1,\dots, \Lambda_{n_p}$ be all the components of 
$\bgres(A_{p,\wedge})$, defined as in \eqref{bgresSectors} via  
an enumeration $a_1^0,\dots,a_{n_p}^0$ of the distinct normalized
coefficients of $A_{p,\wedge}$.
 
We know that $\lambda\in \bgres(A_{p,\wedge})$ is in the spectrum of
$A_{p,\wedge}$ with domain $\Dom_{\gamma_p}$ if and only if it solves 
the equation \eqref{EigenwertEquation}. Since $\nu_p=0$, on each sector 
$\Lambda_\ell$ the condition \eqref{EigenwertEquation} can be reduced to 
an equation of the form $\sum_{j=1}^{n_p} 
d_j\epsilon_j(\Lambda_\ell)/{\sqrt{a_j^0}} = 0$, where the $d_j$ are
constants and the $\epsilon_j$ are the functions from \eqref{SignFunction}.  
Thus the task is to find $d_1,\dots,d_{n_p}\in\C$ such that 
\begin{gather*}
\sum_{j=1}^{n_p} \frac{d_j\epsilon_j(\Lambda_\ell)}{\sqrt{a_j^0}} = 0
\quad \text{for } \ell\in\{j_1,\dots,j_m\}, 
\intertext{and}
\sum_{j=1}^{n_p} \frac{d_j\epsilon_j(\Lambda_\ell)}{\sqrt{a_j^0}}\not= 0
\quad \text{for } \ell\not\in\{j_1,\dots,j_m\}.
\end{gather*}
To this end, consider the system
\begin{equation}\label{SignSystem}
\begin{pmatrix}
\epsilon_1(\Lambda_{1}) & \cdots & \epsilon_{n_p}(\Lambda_{1}) \\
\vdots & & \vdots \\
\epsilon_1(\Lambda_{n_p}) & \cdots & \epsilon_{n_p}(\Lambda_{n_p}) 
\end{pmatrix}
\begin{pmatrix}
y_1 \\ \vdots \\ y_{n_p}
\end{pmatrix} =
\begin{pmatrix}
\delta_1 \\ \vdots \\ \delta_{n_p}
\end{pmatrix}, 
\end{equation}
where $\delta_\ell=0$ for $\ell\in\{j_1,\dots,j_m\}$ and $\delta_\ell=1$
for $\ell\not\in\{j_1,\dots,j_m\}$.  By Lemma~\ref{EpsilonValues}, 
the entries of the matrix $[\epsilon_j(\Lambda_\ell)]_{\ell,j}$ are $1$
on and below the diagonal, and $-1$ above the diagonal. 
Since this matrix is regular, the system \eqref{SignSystem} is solvable. 
Finally, if $(y_1,\dots,y_{n_p})$ is a solution vector of \eqref{SignSystem}, 
then we choose $d_j=y_j \sqrt{a_j^0}$ and use $d_1,\dots,d_{n_p}$ to find 
a corresponding vector $(c_1',\dots,c_{k_p}')\in \C^{k_p}\minus\{0\}$. 
\end{proof} 

%%%%%%%%%%%%%%%%%%%%%%%%%%%%%%%%%%%%%%%%%%%%%%%%%%%%%%%%%%%%%%%%%%%%%
\section{Resolvent decay for the model operator}
\label{Sec-DecModelOperator}

We now analyze the existence of sectors of minimal growth for the model
operator $A_{p,\wedge}$ with an admissible domain $\Dom_{\gamma_p}$.
Obviously, a necessary condition for a closed sector $\Lambda$ to be of
minimal growth for $A_{p,\wedge}$ is that $\Lambda \cap \bgspec(A_{p,\wedge}) = \{0\}$.
Since $\bgres(A_{p,\wedge})$ is a union of open sectors, 
$\Lambda \minus \{0\}$ must be contained in one of these. In fact, 
for every open sector $\Lambda_0 \subset \bgres(A_{p,\wedge})$, either every 
closed subsector $\Lambda\subset\Lambda_0\cup\{0\}$ is a sector of minimal 
growth for $A_{p,\wedge}$ with domain $\Dom_{\gamma_p}(A_{p,\wedge})$, or 
none of them is, see Proposition~\ref{TheoremWedge}.

For $\varrho > 0$ define
\[ \kappa_{\varrho} : \bigoplus_{j=1}^{k_p}L^2(\R_+) \to
\bigoplus_{j=1}^{k_p}L^2(\R_+) \]
by
\begin{equation}\label{kappa}
\kappa_{\varrho}\Bigl(\oplus_{j=1}^{k_p} u_j\Bigr) =
\oplus_{j=1}^{k_p}\big(\varrho^{1/2}u_j(\varrho x)\big).
\end{equation}
This is a strongly continuous unitary one-parameter group.

The spaces $\Dom_{\max}(A_{p,\wedge})$ and $\Dom_{\min}(A_{p,\wedge})$ are 
both $\kappa$-invariant, so $\kappa_{\varrho}$ descends to an action
$$
\kappa_{\varrho} : \Dom_{\max}(A_{p,\wedge})/\Dom_{\min}(A_{p,\wedge}) \to
\Dom_{\max}(A_{p,\wedge})/\Dom_{\min}(A_{p,\wedge}).
$$
Moreover, by means of the map
$$
\Dom_{\gamma_p}(A_{p,\wedge}) \to \Dom_{\gamma_p}(A_{p,\wedge})/
\Dom_{\min}(A_{p,\wedge}) \subset \Sing_p(A_{\wedge}),
$$
the admissible domains are in one-to-one correspondence with 
$k_p$-dimensional subspaces of $\Sing_p(A_{\wedge})$. 
Hence it makes sense to consider the induced flow
$$
\kappa_{\varrho} : 
\Gr_{k_p}(\Sing_p(A_{\wedge})) \to  \Gr_{k_p}(\Sing_p(A_{\wedge})),
\quad \varrho>0,
$$
on the Grassmannian $\Gr_{k_p}(\Sing_p(A_{\wedge}))$ of $k_p$-dimensional
subspaces of $\Sing_p(A_{\wedge})$. The correspondence
$$
\gamma_p \longleftrightarrow \Dom_{\gamma_p}(A_{p,\wedge})
$$
induces an identification $\Gr_{k_p}(\Sing_p(A_{\wedge})) \cong
G_{k_p,2k_p}(\C)$, and so we get the flow
\begin{equation}\label{kappaflowGrass}
\kappa_{\varrho} : G_{k_p,2k_p}(\C) \to G_{k_p,2k_p}(\C), \quad\varrho>0.
\end{equation}
More precisely, if the class $[\gamma_p]\in G_{k_p,2k_p}(\C)$ is 
represented by 
\[ \gamma_p= \begin{pmatrix} C_p & C_p' \end{pmatrix}, \]
then $\kappa_{\varrho}[\gamma_p]$ can be represented by 
\begin{equation*} \label{kapparepresentation}
\begin{pmatrix}
C_p & \varrho^{-1}C_p'
\end{pmatrix}.
\end{equation*}
If $\rk C_p'= \ell$, $0 \leq \ell \leq k_p$, then the matrix
$\begin{pmatrix} C_p & C_p' \end{pmatrix}$ is equivalent to
\begin{equation*} 
\begin{pmatrix}
C_{p,1} & C_{p,1}' \\
C_{p,2} & 0
\end{pmatrix},
\end{equation*}
where $C_{p,1}$ and $C_{p,1}'$ are $(\ell\times k_p)$-matrices, and so
\[
\begin{pmatrix}
C_p & \varrho^{-1}C_p'
\end{pmatrix}
\sim
\begin{pmatrix}
C_{p,1} & \varrho^{-1} C_{p,1}' \\
C_{p,2} & 0
\end{pmatrix}
\sim
\begin{pmatrix}
\varrho C_{p,1} & C_{p,1}' \\
C_{p,2} & 0
\end{pmatrix}.
\]

As a consequence, we obtain the following:

\begin{proposition}\label{LimitingDomains}
For every $[\gamma_p] \in G_{k_p,2k_p}(\C)$, the limit of 
$\kappa_{\varrho}[\gamma_p]$ as $\varrho \to 0$ exists in 
$G_{k_p,2k_p}(\C)$. Moreover,  if the right $(k_p \times k_p)$-block of 
$\gamma_p$ has rank $\ell$, $0 \leq \ell \leq k_p$, then 
$\lim\limits_{\varrho \to 0} \kappa_{\varrho}[\gamma_p]$ can be 
represented by a matrix of the form
\begin{equation} \label{LimitingDomainMatrix}
\begin{pmatrix}
0 & \cdots & 0 & c_{1,1}' & \cdots & c_{1,k_p}' \\
\vdots & & \vdots & \vdots & & \vdots \\
0 & \cdots & 0 & c_{\ell,1}' & \cdots & c_{\ell,k_p}' \\
c_{\ell+1,1} & \cdots & c_{\ell+1,k_p} & 0 & \cdots & 0 \\
\vdots & & \vdots & \vdots & & \vdots \\
c_{k_p,1} & \cdots & c_{k_p,k_p} & 0 & \cdots & 0
\end{pmatrix}.
\end{equation}
\end{proposition}

\medskip
\noindent
The domain $\Dom_{0}(A_{p,\wedge})$ induced by the coupling condition 
$\lim\limits_{\varrho \to 0}\kappa_{\varrho}[\gamma_p]$ will be referred to 
as the \emph{limiting domain} of $\Dom_{\gamma_p}(A_{p,\wedge})$ with respect 
to the $\kappa$-flow. Note that this limiting domain is $\kappa$-invariant.

\begin{lemma}\label{InvEverywhere}
Let $[\gamma_p] \in G_{k_p,2k_p}(\C)$ represent a $\kappa$-invariant domain 
$\Dom_{\gamma_p}$ for $A_{p,\wedge}$. Let $\Lambda_0$ be an open sector in 
$\bgres(A_{p,\wedge})$. Then 
\[ \text{either }\; \Lambda_0 \subset 
   \spec\bigl(A_{p,\wedge}\big|_{\Dom_{\gamma_p}}\bigr) 
   \;\text{ or }\; \Lambda_0 \cap 
   \spec\bigl(A_{p,\wedge}\big|_{\Dom_{\gamma_p}}\bigr) =\varnothing.
\]
Moreover, if $\Lambda_0 \cap \spec\bigl(A_{p,\wedge}\big|_{\Dom_{\gamma_p}}
\bigr)= \varnothing$, then every closed subsector $\Lambda \subset 
\Lambda_0\cup\{0\}$ is a sector of minimal growth for $A_{p,\wedge}$ 
with domain $\Dom_{\gamma_p}$.
\end{lemma}
\begin{proof}
An elementary argument (or directly by Proposition~\ref{LimitingDomains}) 
shows that if $\Dom_{\gamma_p}$ is $\kappa$-invariant, then 
$\gamma_p$ is equivalent to a matrix of the form
\eqref{LimitingDomainMatrix}
for some $0\le \ell\le k_p$.  Let $\lambda\in\Lambda_0\subset
\bgres(A_{p,\wedge})$ and choose the square roots $\sqrt{-\lambda}$ 
and $\sqrt{a_j(0)}$ in such a way that $\Re \bigl(\sqrt{-\lambda}/
\sqrt{a_j(0)}\bigr)>0$.  According to \eqref{AwedgeDet} we now get that 
$\lambda \in \spec\bigl(A_{p,\wedge}\big|_{\Dom_{\gamma_p}}\bigr)$ 
if and only if
\begin{equation*}
(-1)^\ell (\sqrt{-\lambda})^\ell \det\begin{pmatrix}
\frac{c_{1,1}'}{\sqrt{a_1(0)}} & \cdots &
\frac{c_{1,k_p}'}{\sqrt{a_{k_p}(0)}} \\
\vdots & & \vdots \\
\frac{c_{\ell,1}'}{\sqrt{a_1(0)}} & \cdots &
\frac{c_{\ell,k_p}'}{\sqrt{a_{k_p}(0)}} \\
c_{\ell+1,1} & \cdots & c_{\ell+1,k_p} \\
\vdots & & \vdots \\
c_{k_p,1} & \cdots & c_{k_p,k_p}
\end{pmatrix} = 0.
\end{equation*}
This proves the assertion about the spectrum.
It is easy to verify that
\begin{equation}\label{kappaHomogeneity}
A_{p,\wedge}-\varrho^2\lambda = \varrho^2 \kappa_{\varrho}\bigl(
A_{p,\wedge}-\lambda\bigr)\kappa_{\varrho}^{-1}: 
\Dom_{\gamma_p}(A_{p,\wedge})\to \bigoplus_{j=1}^{k_p}L^2(\R_+),
\end{equation}
and thus the operator norm in $\bigoplus_{j=1}^{k_p}L^2(\R_+)$ of
$$
|\lambda|\cdot \bigl(A_{p,\wedge}\big|_{\Dom_{\gamma_p}}-\lambda\bigr)^{-1} 
= \kappa_{|\lambda|^{1/2}}\bigl(A_{p,\wedge}\big|_{\Dom_{\gamma_p}}
-\hat{\lambda}\bigr)^{-1}\kappa_{|\lambda|^{1/2}}^{-1}
$$ 
is $O(1)$ as $|\lambda| \to \infty$, uniformly for $\hat{\lambda} = 
\lambda/|\lambda|$ in compact sets. Note that the group action 
$\kappa_{\varrho}$ is unitary in $\bigoplus_{j=1}^{k_p}L^2(\R_+)$.
This proves the lemma.
\end{proof}
The dilation property \eqref{kappaHomogeneity}, often referred to
as $\kappa$-homogeneity, is systematically used by Schulze~\cite{SchuNH} 
in his theory of algebras of pseudodifferential operators 
on manifolds with singularities.

\medskip
Let $[\gamma_p]\in G_{k_p,2k_p}(\C)$, and let $\gamma_{p,0}$ be a
representative of $\lim_{\varrho\to 0}\kappa_{\varrho}[\gamma_p]$, 
where $\gamma_{p,0}$ is of the form \eqref{LimitingDomainMatrix}.
Let $\Lambda\subset\C$ be any sector with $a_j(0)\notin \Lambda$ for
$j=1,\dots,k_p$.
With $\gamma_p$ and $\Lambda$ we associate the matrix
\begin{equation} \label{SpecMatrix}
S_{\gamma_p,\Lambda} = 
\begin{pmatrix}
\frac{c_{1,1}'}{\sqrt{a_1(0)}} & \cdots &
\frac{c_{1,k_p}'}{\sqrt{a_{k_p}(0)}} \\
\vdots & & \vdots \\
\frac{c_{\ell,1}'}{\sqrt{a_1(0)}} & \cdots &
\frac{c_{\ell,k_p}'}{\sqrt{a_{k_p}(0)}} \\
c_{\ell+1,1} & \cdots & c_{\ell+1,k_p} \\
\vdots & & \vdots \\
c_{k_p,1} & \cdots & c_{k_p,k_p}
\end{pmatrix}
\end{equation}
using the entries of $\gamma_{p,0}$. The roots $\sqrt{a_j(0)}$
are chosen so that, if $\lambda_0\in\Lambda$ and $\sqrt{-\lambda_0}$ is
any fixed square root of $-\lambda_0$, then the real parts of
$\sqrt{-\lambda_0}/\sqrt{a_j(0)}$ all have the same sign.

\begin{proposition}\label{TheoremWedge}
Let $[\gamma_p] \in G_{k_p,2k_p}(\C)$ be an admissible coupling condition at 
$p$. Let $\Lambda_0$ be an open sector in $\bgres(A_{p,\wedge})$.
If $\det S_{\gamma_p,\Lambda_0}\not=0$, then every closed subsector 
$\Lambda\subset \Lambda_0\cup\{0\}$ is a sector of minimal growth for $A_{p,\wedge}$ 
with domain $\Dom_{\gamma_p}$.
\end{proposition}
\begin{proof}
In \cite{GKM1} and \cite{GKM3} it is proved, in the general context of 
elliptic cone operators, that a closed sector $\Lambda\subset 
\bgres(A_{p,\wedge})$ is a sector of minimal growth for $A_{p,\wedge}$ 
with domain $\Dom$ if and only if this is the case for the extension
with the limiting domain(s) with respect to the $\kappa$-flow.  Thus 
the statement follows from Lemma~\ref{InvEverywhere}.
\end{proof}

%%%%%%%%%%%%%%%%%%%%%%%%%%%%%%%%%%%%%%%%%%%%%%%%%%%%%%%%%%%%%%%%%%%%%%
\section{Resolvent decay on the graph}\label{Sec-DecGraph}

Let $A$ be an operator as in \eqref{AGlobal} and \eqref{ALocalEdge}.
We assume that $A$ is elliptic and let $\Dom_{\gamma}$ be a domain 
specified as the kernel of a surjective block-diagonal map
$$
\gamma = \bigoplus\limits_{p \in V}\gamma_p : \bigoplus\limits_{p \in 
V}\Sing_p(A) \to \bigoplus\limits_{p \in V}\C^{k_p},
$$
see \eqref{AssembleDomainAsKernel}. For $\lambda \in\C$ let
$$
\K(\lambda) = \ker(A_{\max} - \lambda) \subset \Dom_{\max}(A).
$$

By Proposition~\ref{InjectiveSurjective} we have that, for every 
$\lambda\in\C$, $A-\lambda$ is injective on $\Dom_{\min}(A)$ and surjective
on $\Dom_{\max}(A)$. Thus Lemma~\ref{AlgebraicIntermediate} gives the 
analogue of \eqref{IntersecCond} for the operator $A$ with domain 
$\Dom_{\gamma}$: An element $\lambda \in \C$ belongs to the spectrum 
of $A_{\Dom_{\gamma}}$ if and only if
\begin{equation}\label{IntersecCondGraph}
\Dom_{\gamma}/\Dom_{\min}(A) \cap
\bigl(\K(\lambda)+\Dom_{\min}(A)\bigr)/\Dom_{\min}(A) \neq \{0\}.
\end{equation}
Note that, while the coupling condition $\gamma$ specifies the space
$\Dom_{\gamma}/\Dom_{\min}(A)$ explicitly, it is in general not easy 
to get a hold on the asymptotics of eigenfunctions in the space $\bigl(
\K(\lambda)+\Dom_{\min}(A)\bigr)/\Dom_{\min}(A)$ for each $\lambda$.
In Section~\ref{Sec-ModelOperator}, it is precisely the knowledge of the 
asymptotics of eigenfunctions for the model operator $A_{p,\wedge}$ what 
allowed us to replace the corresponding condition \eqref{IntersecCond} 
by an explicit determinant condition to decide whether a given 
$\lambda \in \C$ belongs to the spectrum of $A_{p,\wedge}$ with domain 
$\Dom_{\gamma_p}(A_{p,\wedge})$, see Proposition~\ref{SpectrumonWedge}. 

Our focus of interest here, however, is whether a given closed sector
$\Lambda \subset \C$ is a sector of minimal growth for $A$ with domain 
$\Dom_{\gamma}$.
This is an asymptotic question about the resolvent 
rather than a problem for a fixed value of $\lambda$.  Using the 
results for the model operator from Section~\ref{Sec-DecModelOperator}, 
we have the following simple answer:

\begin{theorem}\label{SectorMinGrowthGraph}
Let $\Lambda\subset\C$ be a closed sector.  Assume that in \eqref{ALocalEdge} the condition
\begin{equation}\label{ParameterElliptic} 
a_j(s) \notin \Lambda \text{ for all } s \in [-1,1]
\end{equation}
is satisfied for each $j$.  Let $\gamma$ be an admissible coupling condition.
If for every $p \in V$ the determinant of the matrix
$S_{\gamma_p,\Lambda}$ in \eqref{SpecMatrix}
is nonzero, then $\Lambda$ is a sector of minimal growth for the operator
$$
A_{\Dom_{\gamma}} : \Dom_{\gamma}\subset L^2(\Gamma) \to L^2(\Gamma).
$$
\end{theorem}
\begin{proof}
The hypothesis \eqref{ParameterElliptic} and Lemma~\ref{bgspecwedge}
give $\Lambda \subset \bgres(A_{p,\wedge})$ for every $p \in V$.
Let $\Dom_{\gamma_p}$ be the domain for $A_{p,\wedge}$ such that
\begin{equation*} 
   \Dom_{\gamma_p}/\Dom_{\min}(A_{p,\wedge}) =
   \theta_{p}\Big(\pi_p \Dom_{\gamma}/\Dom_{\min}(A)\Big),
\end{equation*} 
where $\theta_p$ is the map \eqref{thetap} and $\pi_{p}:\Sing(A)\to
\Sing_p(A)$ is the canonical projection. 

The condition on the determinant of \eqref{SpecMatrix} ensures,
via Proposition~\ref{TheoremWedge}, that $\Lambda$ is a sector of 
minimal growth for $A_{p,\wedge}$ with domain
$\Dom_{\gamma_p}$. Necessarily, $\Lambda$ is then
a sector of minimal growth for the direct sum
$A_{\wedge} = \bigoplus_{p \in V}A_{p,\wedge}$ 
with the induced domain. By \cite[Theorem~6.9]{GKM2}, 
$\Lambda$ is a sector of minimal growth for $A_{\Dom_{\gamma}}$.
\end{proof}

In Example~\ref{deltaConditions} we discussed $\delta$-type
conditions $\gamma_p$ for the model operator $A_{p,\wedge}$ at a vertex $p
\in V$. Correspondingly, we can define an associated admissible coupling
condition $\gamma$ for the operator $A$ on $\Gamma$ by imposing
$\delta$-type conditions at all vertices (formally given by the same
matrix representations as for $A_{p,\wedge}$). Thus, if $\Dom_{\gamma}$
is the domain associated with $\gamma$, every $u \in \Dom_{\gamma}$ is
continuous on $\Gamma$ and the asymptotic coefficients $\beta_q$ from
\eqref{aSingFunction} at the boundary points $q \in p$ are all related 
via the last row in the matrix \eqref{deltaConditionsMatrix}.

In view of Theorem~\ref{SectorMinGrowthGraph},
Proposition~\ref{TheoremWedge}, and Proposition~\ref{SectorsSpectrum},
we have that for suitable choices of $\gamma_p$ it is possible
to obtain resolvent decay for $A_{\Dom_{\gamma}}$ along certain closed
sectors in $\C$ (determined by the leading coefficients of $A$).

A partial converse of Theorem~\ref{SectorMinGrowthGraph} follows from
\cite[Theorem~5.5]{GKM3}. The converse holds, in particular, if we consider the
operator $A$ given by various multiples of the Laplacian on the edges 
$E_j$ of the graph. In this case, we can choose specific $\delta$-type 
conditions $\gamma$ such that the resolvent of $A_{\Dom_{\gamma}}$ decays 
along certain sectors, but it does not decay along any ray in the 
complementary directions.

%%%%%%%%%%%%%%%%%%%%%%%%%%%%%%%%%%%%%%%%%%%%%%%%%%%%%%%%%%%%%%%%%%%%%
\section{Regular singular operators on graphs}\label{Sec-ConeOperators}

In this section we briefly survey our main result as regards the existence 
of sectors of minimal growth for regular singular operators on a 
graph $\Gamma = G/\sim$. For a thorough study of closed extensions
and resolvents of general elliptic operators on manifolds with conical
singularities the reader is referred to \cite{GKM1, GKM2, GKM3}.  

Consider operators
$$
A : C_c^{\infty}(\open \Gamma) \to C_c^{\infty}(\open \Gamma)
$$
of order $m > 0$ which along each edge $E_j\cong [-1,1]$ are of the form  
\begin{equation*}
A_j = \sum_{\nu=0}^m a_{j,\nu} \frac{1}{(1-s^2)^{m-\nu}} D_s^\nu 
\end{equation*}
with $a_{j,\nu}$ smooth on $[-1,1]$.
Assume that $A$ is elliptic, i.e., the leading coefficients
$a_{j,m}$ do not vanish.
Using the defining function $x$ we can write $A$ in the form
\begin{equation}\label{coneoperator}
A = x^{-m}\sum\limits_{\nu=0}^m a_{q,\nu}(x)(xD_x)^{\nu}
\end{equation}
near each boundary point $q \in \partial G$.  

The unbounded operator
$$
A : C_c^{\infty}(\open \Gamma) \subset L^2(\Gamma) \to L^2(\Gamma)
$$
is closable and the domain $\Dom$ of any closed extension of $A$ 
contains $\Dom_{\min}(A)$ and is contained in $\Dom_{\max}(A)$.
More specifically, an asymptotic analysis that involves the indicial 
polynomial of $A$ at the boundary and the Mellin transform reveals that
$$
\Dom_{\max}(A)/\Dom_{\min}(A) \cong \bigoplus\limits_{p \in V}
\Sing_p(A),
$$
where $\Sing_p(A)$ is a finite dimensional space of singular
functions similar to \eqref{aSingFunction}, \eqref{EndPointDecomposition}, and
\eqref{VertexDecomposition2}. However, the structure of this
space is in general more complicated than that of the operators analyzed 
in the previous sections.

A coupling condition on the graph determines a domain $\Dom_{\min}(A) \subset \Dom \subset \Dom_{\max}(A)$ with the property that
$$
\Dom/\Dom_{\min}(A) = \bigoplus\limits_{p \in V}\pi_p
\bigl(\Dom/\Dom_{\min}(A)\bigr),
$$
where $\pi_p : \Dom_{\max}(A)/\Dom_{\min}(A) \to \Sing_p(A)$ is the
canonical projection.

As discussed in the previous sections, the interplay between the domains of
$A$ and the domains of the model operator $A_\wedge$ is of
central interest for their spectral analysis.  Thus, as in 
Section~\ref{Sec-ModelOperator}, we consider the model operator
$$
A_{p,\wedge} : \bigoplus_{q\in p}C_c^{\infty}(\R_+) \to
\bigoplus_{q\in p}C_c^{\infty}(\R_+)
$$
associated with $A$ at the vertex $p$.
For $q \in p$, and $A$ as in \eqref{coneoperator}
we have $A_{p,\wedge} = x^{-m}\sum_{\nu=0}^m a_{q,\nu}(0)(xD_x)^{\nu}$
defined on the $\R_+$-axis associated with $q$. Note that $A_{p,\wedge}$ 
is a diagonal operator.

The domains of the closed extensions of $A_{p,\wedge}$ in
$\bigoplus_{q\in p}L^2(\R_+)$ are intermediate spaces
$\Dom_{\min}(A_{p,\wedge}) \subset \Dom(A_{p,\wedge}) \subset
\Dom_{\max}(A_{p,\wedge})$, and there is a natural (depending
on the boundary defining function) isomorphism
$$
 \theta_p: \Sing_p(A) \to 
 \Dom_{\max}(A_{p,\wedge})/\Dom_{\min}(A_{p,\wedge}).
$$
This linear map can  be constructed explicitly and is determined by a
simple algorithm of finitely many steps, for details see \cite{GKM1}.  
The map \eqref{thetap} is an example.
Thus, with a given domain $\Dom$ corresponding to a coupling 
condition on the graph, there is a domain $\Dom_{p,\wedge}$ for 
$A_{p,\wedge}$ given by the relation
\begin{equation}\label{AssocDomain}
\Dom_{p,\wedge}/\Dom_{\min}(A_{p,\wedge}) =
\theta_p(\pi_p(\Dom/\Dom_{\min}(A))).
\end{equation}

As shown in Sections~\ref{Sec-DecModelOperator} and 
\ref{Sec-DecGraph}, it is crucial to look 
at the flow induced by dilation of the domain $\Dom_{p,\wedge}$.  
To this end, let
$$
\kappa_{\varrho} : \bigoplus_{q\in p}L^2(\R_+) \to
\bigoplus_{q\in p}L^2(\R_+)
$$
be the group action from \eqref{kappa}. Since both domains
$\Dom_{\min}(A_{p,\wedge})$ and $\Dom_{\max}(A_{p,\wedge})$ are
$\kappa$-invariant, we get an induced action on the quotient, which
in turn gives rise to a (multiplicative) flow
$$
\kappa_{\varrho} : \Dom(A_{p,\wedge})/\Dom_{\min}(A_{p,\wedge}) \mapsto
\kappa_{\varrho}\bigl(\Dom(A_{p,\wedge})\bigr)/\Dom_{\min}(A_{p,\wedge}),
\quad \varrho > 0,
$$
on the Grassmannians of subspaces of
$\Dom_{\max}(A_{p,\wedge})/\Dom_{\min}(A_{p,\wedge})$.
We let
\begin{align*}
\Omega^{-}\bigl(\Dom_{p,\wedge}\bigr) = \{\Dom_0 \st 
& \;\exists\, 0<\varrho_k \to 0 \textup{ such that } \\
& \lim\limits_{k \to \infty}\kappa_{\varrho_k}\bigl(
\Dom_{p,\wedge}/\Dom_{\min}(A_{p,\wedge})\bigr) =
\Dom_0/\Dom_{\min}(A_{p,\wedge})\},
\end{align*}
where the convergence refers to the Grassmannian of subspaces of a
fixed dimension in $\Dom_{\max}(A_{p,\wedge})/\Dom_{\min}(A_{p,\wedge})$.

\begin{theorem}\label{GeneralTheorem}
Let $\Lambda \subset \C$ be a closed sector. Assume that on each edge
$E_j$ the leading coefficient $a_{j,m}(s)$ of $A$ satisfies
\[ a_{j,m}(s),\; (-1)^m a_{j,m}(s) \notin \Lambda
   \text{ for all } s\in[-1,1]. \]
Let $\Dom$ be a domain for $A$ associated with a coupling condition on 
a graph $\Gamma$ with vertices $V$. For $p \in V$ let $\Dom_{p,\wedge}$ 
be the domain given by \eqref{AssocDomain}. If for every $p$ and every 
domain $\Dom_0 \in \Omega^{-}\bigl(\Dom_{p,\wedge}\bigr)$ the family
\begin{equation*}
A_{p,\wedge} - \lambda : \Dom_0\subset L^2(\Gamma) \to L^2(\Gamma)
\end{equation*}
is invertible for all $\lambda\in \Lambda$ with $|\lambda| = 1$, 
then $\Lambda$ is a sector of minimal growth for the operator $A$ 
with domain $\Dom$.
\end{theorem}

As our arguments from the previous sections indicate, in order to 
understand the spectral properties of $A$ on the graph $\Gamma$, it is 
essential to analyze the asymptotics of eigenfunctions of the model 
operator $A_{p,\wedge}$ at each vertex $p$.  Since, in the case at hand, 
the $A_{p,\wedge}$ are ordinary differential operators on half-lines with 
regular singular points located at the origin, their eigenfunctions are 
special functions whose asymptotic behavior can sometimes be found
in the existing literature.

The class of second order differential operators with a Coulomb type 
potential, as considered in this paper, constitutes a special case of
the situation above, and Theorem~\ref{SectorMinGrowthGraph} is an 
instance of Theorem~\ref{GeneralTheorem}.  In this case, 
the operators $A_{p,\wedge}$ have a simple structure that makes it 
even possible to compute their eigenfunctions explicitly.  We also 
have that the limiting set $\Omega^{-} \bigl(\Dom_{p,\wedge}\bigr)$ 
consists of a single domain, see Proposition~\ref{LimitingDomains}. 
Moreover, using explicit matrix representations for the coupling
conditions, the invertibility condition for $A_{p,\wedge}-\lambda$ over 
a sector $\Lambda$ can be expressed as a simple determinant condition, 
see Lemma~\ref{InvEverywhere} and Proposition~\ref{TheoremWedge}.

The simplicity of all these very explicit results is certainly not 
representative for the case of general regular singular operators. 
However, they serve to illustrate the strengths of our approach and 
give a new perspective for the study of resolvents and evolution
equations on graphs.

%%%%%%%%%%%%%%%%%%%%%%%%%%%%%%%%%%%%%%%%%%%%%%%%%%%%%%%%%%%%%%%%%%%%%%

\end{document}